\documentclass[a4paper,11pt]{article}
\usepackage[T1]{fontenc}
\usepackage[utf8]{inputenc}
\usepackage{fullpage}

\usepackage[all=normal,bibliography=tight]{savetrees}

\usepackage{amsmath,amssymb,amsfonts,amsthm}
\usepackage[tight]{subfigure}
\usepackage{tikz}
\usepackage{todonotes}
\usepackage[absolute]{textpos}
\usepackage{enumerate}
\usepackage{paralist}
\usepackage{hyperref}
\usepackage[noabbrev,capitalize]{cleveref}
\usepackage{makecell}

\usetikzlibrary{shapes,snakes}

\newtheorem{lemma}{Lemma}[section]

\newtheorem{theorem}[lemma]{Theorem}
\newtheorem{claim}[lemma]{Claim}
\theoremstyle{definition}
\newtheorem{definition}[lemma]{Definition}

\usepackage{etoolbox}

\newcommand{\Oh}[0]{{\mathcal{O}}}

\newcommand{\LL}{\mathcal{L}}
\newcommand{\RLL}{\overleftarrow{\LL}}
\newcommand{\Kk}[0]{\mathcal{K}}
\newcommand{\Ss}[0]{\mathcal{S}}
\newcommand{\Zz}[0]{\mathcal{Z}}

\newcommand{\Ww}[0]{\mathcal{W}}

\DeclareMathOperator{\oc}{oc}

\newcommand{\df}[0]{\ensuremath{:=}}

\newcommand{\cycles}{\mathcal{C}}

\newcommand{\IntGraph}[1]{\mathrm{Int}(#1)}

\makeatletter
\newcommand\footnoteref[1]{\protected@xdef\@thefnmark{\ref{#1}}\@footnotemark}
\makeatother

\begin{document}
\title{Constant congestion brambles in directed graphs%
\thanks{This research is part of projects that have received funding from the
European Research Council (ERC) under the European Union's Horizon
2020 research and innovation programme Grant Agreement 714704.

T.M.~completed a part of this work while being supported by a postdoctoral fellowship at the Simon Fraser University through NSERC grants R611450 and R611368.

M.S.~completed a part of this work while being supported by Alexander von Humboldt foundation.}
}

\author{Tom\'{a}\v{s} Masa\v{r}\'{i}k\footnote{ Faculty of Mathematics, Informatics and Mechanics, University of Warsaw, Poland \newline and Department of Mathematics, Simon Fraser University, BC, Canada, \texttt{masarik@kam.mff.cuni.cz}} 
    \and Marcin Pilipczuk\footnote{Faculty of Mathematics, Informatics and Mechanics, University of Warsaw, Poland, \texttt{malcin@mimuw.edu.pl}}
    \and Pawe\l{} Rz\k{a}\.{z}ewski\footnote{Faculty of Mathematics and Information Science, Warsaw University of Technology, Warsaw, Poland\newline and Faculty of Mathematics, Informatics and Mechanics, University of Warsaw, Poland, \texttt{p.rzazewski@mini.pw.edu.pl}}
    \and Manuel Sorge\footnote{Faculty of Mathematics, Informatics and Mechanics, University of Warsaw, Poland\newline and TU Wien, Faculty of Informatics, Vienna, Austria, \texttt{manuel.sorge@mimuw.edu.pl}}
}
\date{}

\maketitle

\begin{abstract}
The Directed Grid Theorem, stating that there is a function $f$ such that a directed graphs of directed treewidth at least $f(k)$ contains a directed grid of size at least~$k$ as a butterfly minor, after being a conjecture for nearly 20 years, has been proven in 2015 by Kawarabayashi and Kreutzer.
However, the function $f$ obtained in the proof is very fast growing.

In this work, we show that if one relaxes \emph{directed grid} to \emph{bramble of constant congestion}, one can obtain a polynomial bound. More precisely, we show that for every $k \geq 1$ there exists $t = \Oh(k^{48} \log^{13} k)$ such that every directed graph of directed treewidth at least $t$
contains a bramble of congestion at most $8$ and size at least $k$. 

%%% Local Variables:
%%% mode: latex
%%% TeX-master: "main"
%%% End:

\end{abstract}

\begin{textblock}{20}(0, 11.5)
\includegraphics[width=40px]{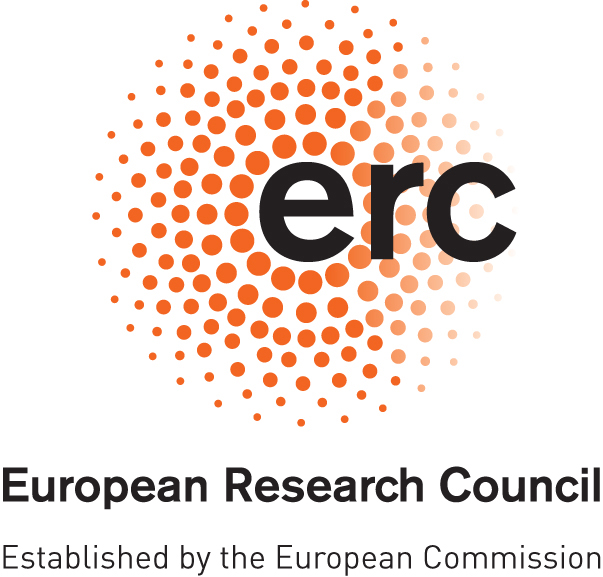}%
\end{textblock}
\begin{textblock}{20}(-0.25, 11.9)
\includegraphics[width=60px]{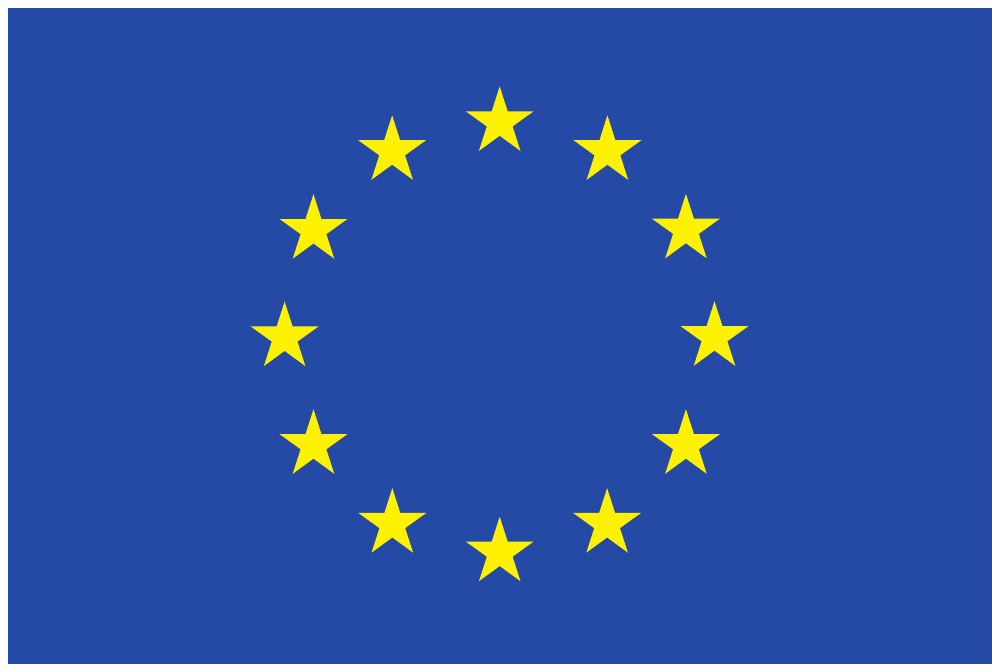}%
\end{textblock}

\section{Introduction}
The Grid Minor Theorem, proven by Robertson and Seymour~\cite{RobertsonS86}, is arguably one of the most important structural characterizations of treewidth. 
Informally speaking, it asserts that a \emph{grid minor} is a canonical obstacle to small treewidth: a graph of large treewidth necessarily contains a big grid as a minor. 
The relation of ``large'' and ``big'' in this statement, being non-elementary in the original proof, after a series of improvements has been proven to
be a polynomial of relatively small degree:
\begin{theorem}[\cite{ChuzhoyT21}]
For every $k \geq 1$ there exists $t = \Oh(k^9 \mathrm{polylog} k)$ such that
every graph of treewidth at least $t$ contains a $k \times k$ grid as a minor.
\end{theorem}
In the mid-90s, Johnson, Robertson, Seymour, and Thomas~\cite{JohnsonRST01}
proposed an analog of treewidth for directed graphs, called \emph{directed treewidth}, 
and conjectured an analogous statement (with the appropriate notion of a directed grid). 
After nearly 20 years, the Directed Grid Theorem was proven in 2015 by Kawarabayashi and Kreutzer~\cite{DBLP:conf/stoc/KawarabayashiK15}.
However, their proof yields a very high dependency between the required directed treewidth bound and the promised size of the directed grid.

While searching for better and better bounds for (undirected) Grid Minor Theorem, researchers
investigated relaxed notions of a grid (e.g. \cite{kawarabayashi_excluded_2014a}.
In some sense, the ``most relaxed'' notion of a grid is a \emph{bramble}: a family $\mathcal{B}$
of connected subgraphs of a given graph such that every two $B_1,B_2 \in \mathcal{B}$ either
share a vertex, or there exists an edge with one endpoint in $B_1$ and one endpoint in $B_2$. 
Brambles can be large; the notion of complexity of a bramble is its \emph{order}: the minimum
size of a vertex set that intersects every element of a bramble. 
We also refer to the \emph{size} of a bramble as the number of its elements and the \emph{congestion}
of a bramble as a maximum number of elements that contain a single vertex; note that the size of
the bramble is bounded by the product of its order and congestion. 

To link brambles and grids, first note that a $k \times k$ grid contains a simple bramble of order $k$ and size $k^2$: the elements
of a bramble are subgraphs consisting of $i$-th row and $j$-th column of the grid for every $1 \leq i,j \leq k$. 
If one wants a bramble of congestion $2$, a bramble of size $k$ whose elements are subgraphs consisting of the $i$-th row and $i$-th column of the grid for every $1 \leq i \leq k$
is of order $\lceil k/2 \rceil$. 
In the other direction, brambles of small congestion can replace grids 
if one wants to use a grid as an object that allows arbitrary interconnections of small congestion between different pairs of vertices on its boundary.
Such a usage appears e.g.\ in arguments for the Disjoint Paths problem~(cf. \cite{ChekuriC13,ChekuriEP18,ChekuriKS05,EdwardsMW17}).

Surprisingly, as proven by Seymour and Thomas~\cite{SeymourT93}, brambles form a dual object
tightly linked to treewidth: the maximum order of a bramble in a graph is \emph{exactly} 
the treewidth of the graph plus one. 
However, as shown by Grohe and Marx~\cite{GroheM09} and sharpened by Hatzel et al.~\cite{HatzelKPS20}, brambles of high order may need to have exponential size:
while a graph of treewidth $k$ neccessarily contains a bramble of order
$\widetilde{\Omega}(\sqrt{k})$ of congestion 2 (and thus of size linear in their order),
there are classes of graphs (e.g., constant-degree expanders) where
for every $0 < \delta < 1/2$ any bramble of order $\widetilde{\Omega}(k^{0.5+\delta})$
requires size exponential in roughly $k^{2\delta}$. 
Here, the notation $\widetilde{\Omega}$ and $\widetilde{\Oh}$ omits polylogarithmic factors.

A slightly more organized bramble of congestion $2$, namely two families of vertex-disjoint
paths with an intersection graph containing a large clique minor
(with size bound of quartic dependence on the treewidth), has been shown to exist in undirected graphs
by Reed and Wood~\cite{DBLP:journals/ejc/ReedW12}. 

In directed graphs, the notion of a bramble naturally generalizes to a family of strongly connected
subgraphs such that every two subgraphs either intersect in a vertex, or the graph contains
an arc with a tail in the first subgraph and a head in the second and
an arc with a tail in the second subgraph and a head in the first.
The order of a directed bramble is defined in the same way as in undirected graphs.
While we no longer have a tight relation between directed treewidth and maximum order
of a directed bramble, these two graph parameters are within a constant factor of each other,
as shown by Reed~\cite{Reed99}. 
However, the lower bound of Grohe and Marx~\cite{GroheM09} also applies to directed graphs:
there are digraph families 
where a graph of directed treewidth $k$ contains only brambles of order $k^{0.5+\delta}$
of exponential size, for any $0 < \delta < 0.5$. 

Hence, it is natural to ask what order of a bramble of constant congestion we can expect
in a directed graph of directed treewidth $t$. The lower bound of Grohe and Marx shows that
we cannot hope for a better answer than $\widetilde{\Oh}(\sqrt{t})$. 
Since a directed grid contains a bramble of congestion $2$ and order linear in the size
of the grid, the Directed Grid Theorem implies that for every $k \geq 1$ there exists $t=t(k)$
such that directed treewidth at least $t$ guarantees an existence of a bramble of order $k$
and congestion $2$. However, the function $t=t(k)$ stemming from the proof
of Kawarabayashi and Kreutzer~\cite{DBLP:conf/stoc/KawarabayashiK15} is very fast-growing.
Similarly, a half-integral variant of the Directed Grid Theorem~\cite{kawarabayashi_excluded_2014a} could be used to obtain a bramble of order~$k$ and congestion~$4$ but also there the function $t=t(k)$ is very fast-growing.

In this work, we show that this dependency can be made polynomial, 
   if we are satisfied with slightly larger congestion.

\begin{theorem}\label{thm:main}
For every $k \geq 1$ there exists $t = \Oh(k^{48} \log^{13} k)$
such that every directed graph of directed treewidth at least $t$
contains a bramble of congestion at most $8$ and size at least $k$.
\end{theorem}

So far, similar bounds were known only for planar graphs, where Hatzel, Kawarabayashi,
  and Kreutzer showed a polynomial bound (with degree $6$ of the polynomial) for the Directed Grid Theorem~\cite{DBLP:conf/soda/HatzelKK19}. Decreasing the congestion in~\cref{thm:main}, ideally to $2$, even at the cost of higher polynomial dependency of $t$ and $k$, remains an interesting open problem.
  Optimizing the parameters 
  in the other direction would also be interesting: for all we know, obtaining the
  dependency $t = \widetilde{\Oh}(k^2)$ for constant congestion may be possible. 

On the technical level, the proof of~\cref{thm:main} borrows a number of tools from 
previous works. From Reed and Wood~\cite{DBLP:journals/ejc/ReedW12}, we borrow the idea
of using Kostochka-Thomason degeneracy bounds for graphs excluding a minor~\cite{kostochka_lower_1984,thomason_1984} to ensure the existence of a large clique minor in an intersection graph of a family of strongly connected subgraphs, if it turns out to be dense (which immediately gives a desired bramble). We also use their Lov\'{a}sz Local Lemma-based argument to find a large independent set in a multipartite graph of low degeneracy. 
Similarly as in the proof of Directed Grid Theorem~\cite{DBLP:conf/stoc/KawarabayashiK15}
and in its planar variant~\cite{DBLP:conf/soda/HatzelKK19},
we start from the notion of a path system and its existence (with appropriate parameters) in graphs
of high directed treewidth. 
Finally, from our recent proof of half- and quarter-integral directed Erd\H{o}s-P\'{o}sa property~\cite{DBLP:journals/corr/abs-1907-02494,DBLP:conf/esa/MasarikMPRS19}, 
we reuse their partitioning lemma, allowing us to find a large number of closed walks with small congestion.
On top of the above, compared to~\cite{DBLP:conf/soda/HatzelKK19} and~\cite{DBLP:journals/corr/abs-1907-02494,DBLP:conf/esa/MasarikMPRS19}, the proof of~\cref{thm:main} offers a much more elaborate analysis of the studied path system, allowing us to find the desired bramble. 

\paragraph{Organization.}
We collect the formal statements of results from previous work in \cref{sec:prelim}.
In \cref{sec:tools} we gather tools that show how to obtain a low congestion bramble in various special situations and we show how to obtain some intermediate structures.
In \cref{sec:main} we then show how to combine all the tools to prove \cref{thm:main}.
  
%%% Local Variables:
%%% mode: latex
%%% TeX-master: "main"
%%% End:

\section{Preliminaries}\label{sec:prelim}
For integers $n \in \mathbb{N}$ we use $[n]$ to denote $\{1, 2, \ldots, n\}$.
\paragraph{Basics.}
Let $G$ be a directed graph.
A \emph{walk} in $G$ is a sequence~$W$ of vertices such that for each pair of consecutive vertices~$u, v$ in $W$ there is an arc $(u, v)$ in $G$.
A walk is \emph{closed} if it starts and finishes with the same vertex.
Let $W$ be a walk.
We denote by $V(W)$ the set of vertices that occur in the sequence~$W$.
A \emph{subwalk} of $W$ is a segment of~$W$, that is, a subsequence of consecutive elements.
The number of \emph{occurrences} of a vertex~$v$ in $W$, denoted by $\oc(v,W)$, is the number of times it occurs in the sequence~$W$; if $W$ is closed and $v$ is its starting vertex, then it is the number of times $v$ occurs in the sequence~$W$ minus one.
The \emph{length} of a walk~$W$ is the sum of the numbers of occurrences of the vertices in~$V(W)$.

\begin{definition}[Congestion, Overlap]
Let $\Ww$ be a family of walks in $G$
and $\Ss$ be a family of subsets of $V(G)$. We define:
\begin{align*}
\text{overlap}(\Ww)\df & \max_{v\in V(G)}{\sum_{W\in \Ww} \oc(v,W), }\\
\text{congestion}(\Ss)\df & \max_{v\in V(G)}{|\{S\in\Ss \mid v\in S\}|}.
\end{align*}
For a set of walks $\Ww$, its congestion is the congestion of the family $\{ V(W) \mid W \in \Ww\}$.
\end{definition}

\paragraph{Linkages, path systems, minors.}

For $A,B \subseteq V(G)$, such that $|A|=|B|$, a \emph{linkage} from $A$ to $B$ in~$G$ is a set of~$|A|$ pairwise vertex-disjoint paths in~$G$, each with a starting vertex in $A$ and ending vertex in $B$.
A set $X \subseteq V(G)$ is \emph{well-linked} if for every $A,B \subseteq X$, s.t. $|A| = |B|$ there are $|A|$ vertex-disjoint $A$-$B$-paths in $G - (X \setminus (A \cup B))$.

\begin{definition}[Path system]\label{def:path-system}
Let $a, b \in \mathbb{N}$. An \emph{$(a,b)$-path system} $(P_i,A_i,B_i)_{i=1}^a$ consists of
\begin{itemize}
\item vertex-disjoint paths $P_1,P_2,\ldots,P_a$, and
\item for every $i \in [a]$, two sets $A_i,B_i \subseteq V(P_i)$, each of size 
$b$, such that every vertex of $B_i$ appears on $P_i$ later than all vertices of $A_i$,
\end{itemize}
such that $\bigcup_{i=1}^a A_i \cup B_i$ is well-linked in $G$.
\end{definition}

In this work, we do not need the exact (and involved) definition of directed treewidth; instead, we immediately jump to path systems via the following lemma. 

\begin{lemma}[{Kawarabayashi, Kreutzer~\cite{DBLP:conf/stoc/KawarabayashiK15,DBLP:journals/corr/KawarabayashiK14} (implicit), see also \cite[Lemma~7]{DBLP:journals/corr/abs-1907-02494}}]\label{lem:KK}
There exists a constant $c_{\textsf{KK}}$ such that for every two integers $a,b \geq 1$
every directed graph $G$ of directed treewidth at least $c_{\textsf{KK}} \cdot a^2 b^2$ contains
an $(a,b)$-path system.
\end{lemma}

The \emph{average degree} of a graph with $n$ vertices and $m$ edges is $2m/n$.
We say that a graph is \emph{$d$-degenerate} if every subgraph of $G$ contains a vertex of degree at most $d$.
The \emph{degeneracy} of a graph~$G$ is minimum $d$ such that $G$ is $d$-degenerate. 
Observe that, for every graph~$G$, we have $\Delta_a(G) \leq 2d(G)$, where $\Delta_a(G)$ is $G$'s average degree and $d(G)$ is $G$'s degeneracy.

\begin{theorem}[{Kostochka~\cite[Theorem~1]{kostochka_lower_1984}, Thomason~\cite[Theorem]{thomason_1984} (restated)}]\label{thm:kostochka}
There exists a constant $c_{\textsf{KT}}$, such that for every $a \geq 2$,
every undirected graph $G$ with degeneracy at least $c_{\textsf{KT}} \cdot a \cdot \sqrt{\log a}$ 
contains $K_a$ as a minor.
\end{theorem}

\begin{lemma}[Reed and Wood~{\cite[Lemma~4.3]{DBLP:journals/ejc/ReedW12}}]\label{lem:LLL}
Let $r$ be an integer with $r \ge 2$, $d$ be a positive real, and $H$ be an $r$-colored graph with color classes $V_1,\ldots,V_r$, such that for every $i \in [r]$ it holds that $|V_i|\ge 4e(r-1)d$  and for every $i \neq j$ the graph $H[V_i\cup V_j]$ is $d$-degenerate. Then there exists an independent set $\{x_1,\ldots,x_r\}$ such that $x_i\in V_i$ for every $i \in [r]$.
\end{lemma}

For a family $\Ss$ of sets, its \emph{intersection graph}, denoted by $\IntGraph{\Ss}$, has vertex set $\Ss$, and two distinct sets $S_1,S_2 \in \Ss$ are adjacent in $\IntGraph{\Ss}$ if $S_1 \cap S_2 \neq \emptyset$.
If $\mathcal{W}$ is a set of walks, then the \emph{intersection graph} $\IntGraph{\mathcal{W}}$ is defined as $\IntGraph{\{V(W) \mid W \in \mathcal{W}\}}$.

\section{Tools}\label{sec:tools}
We now gather the new tools that we need in the main proof.
The general setting is that, if the directed treewidth of our graph is large enough, then there is a path system (\cref{def:path-system}) containing a large number of sets $A_i, B_i$ and large linkages between them.
We then distinguish several cases for sets of pairs of linkages and the densities of the intersection graphs of the paths in these linkages.
We end up with three fundamental scenarios, that each allow us to define a desired bramble.
How the brambles are obtained in these scenarios is shown in \cref{sec:brambles}.

In \cref{sec:threaded} we derive a set of tools that allow us to partition paths in the sets of linkages mentioned above in such a way as to keep both the congestion low and the intersection graphs of the parts sufficiently dense.

\subsection{Extracting a bramble}\label{sec:brambles}
\begin{lemma}[Dense winning scenario]\label{lem:dense}
Let $c_{\textsf{KT}}$ be the constant from~\cref{thm:kostochka}.
If a graph $G$ contains a family $\Ww$  of closed walks of congestion $\alpha$,
whose intersection graph is not $c_{\textsf{KT}} \cdot d \cdot \sqrt{\log d}$-degenerate, then
$G$ contains a bramble of congestion $\alpha$ and size $d$.
\end{lemma}
\begin{proof}
Since $\IntGraph{\Ww}$ is not $c_{\textsf{KT}} \cdot d \cdot \sqrt{\log d}$-degenerate, by~\cref{thm:kostochka}, it contains a $K_d$ minor $\Kk$.
Since each branch set of $\Kk$ induces a connected subgraph of $\IntGraph{\Ww}$, and each vertex of $\IntGraph{\Ww}$ is a closed walk in $G$, we obtain that each branch set of $\Kk$ corresponds to a strongly connected subgraph of $G$.
Furthermore, since between any branch sets of $\Kk$ there is an edge in $\IntGraph{\Ww}$, we conclude that the subsets of $V(G)$ corresponding to branch sets of $\Kk$ form a bramble in $G$.
The congestion bound follows clearly from the fact that the congestion of $\Ww$ is~$\alpha$.
\end{proof}

\begin{lemma}[Sparse winning scenario]\label{lem:sparse1}
There is an absolute constant $c$ with the following property.
Let $a > 1$, $b \geq 1$, and let $(P_i,A_i,B_i)_{i=1}^a$ be an $(a,b)$-path system in $G$.
Let $\mathcal{I}$ be a subset of $[a] \times [a] \setminus \{(i,i) \mid i \in [a]\}$, such that $|\mathcal{I}| \geq 0.6 \cdot a(a-1)$.
Assume that for every $(i,j) \in \mathcal{I}$ we have a path $P_{i,j}$ from $B_i$ to $A_j$
such that $\{P_{i,j} \mid (i,j) \in \mathcal{I}\}$ is of congestion at most $\alpha$.
Then $G$ contains a bramble of congestion at most $2+2\alpha$ and size at least  $c \cdot  \left(\frac{a^{1/2}}{\log^{1/4} a}\right)$.
\end{lemma}
\begin{proof}
Consider a graph $H$ with vertex set $[a]$ and $ij \in E(H)$ if both $(i,j) \in \mathcal{I}$ and $(j,i) \in \mathcal{I}$.
Since $|\mathcal{I}|\geq 0.6 \cdot a(a-1)$, we have $|E(H)| \geq 0.1 \cdot \binom{a}{2}$.
By~\cref{thm:kostochka}, $H$ contains a clique minor of size $p \geq c' \cdot a / \sqrt{\log a}$, where $c'$ is an absolute constant.
Without loss of generality, assume that $p = \binom{q}{2}$ for some integer $q \geq  c \cdot a^{1/2} / \log^{1/4} a$, where $c$ is a constant.
Let $(B_{x,y})_{\{x,y\}\in \binom{[q]}{2}}$ be the family of branch sets of the clique minor of size $p$ in $H$.
Observe that for every $x \in [q]$, the subgraph of $H$ induced by $\bigcup_{y \in [q] \setminus \{x\}} B_{x,y}$ is connected, let $T_x$ be its spanning tree.
Note that for every two distinct $x,y \in [q]$, the trees $T_x$ and $T_y$ intersect in $B_{x,y}$. 
On the other hand, every vertex and every edge of $H$ is contained in at most two trees $T_x$.

For every $i \in [a]$, let $e_i$ be the last edge of $P_i$, whose tail is in $A_i$.
Note that $e_i$ is well-defined, as the set $B_i$ follows $A_i$ on $P_i$, see~\cref{fig:walkWe}~(left).
For every edge $e = ij \in E(H)$, let $W_e$ be a closed walk in $G$ obtained as follows. We start with $P_{i,j}$, and then we follow $P_j$ until we arrive at the starting vertex of $P_{j,i}$. Then we follow $P_{j,i}$, and then $P_i$ until we close the walk, see~\cref{fig:walkWe}~(right).
Note that $W_e$ contains both $e_i$ and $e_j$.
For every $x \in [q]$, define a subgraph $G_x$ of $G$ as the union of all walks $W_e$ for all $e \in E(T_x)$.
Since for every $e = ij \in E(H)$, the walk $W_e$ contains $e_i$ and $e_j$, and $T_x$ is connected, the graph $G_x$
is strongly connected and contains all edges $e_i$ for $i \in V(T_x)$. 
Thus, since every two trees $T_x$ and $T_y$ intersect in $B_{x,y}$,
the family $(G_{x})_{x \in [q]}$ is a bramble of size $q \geq c \cdot  \left(\frac{a^{1/2}}{\log^{1/4} a}\right)$ in $G$.

Now let us argue that the congestion of the constructed bramble is at most $2\alpha+2$.
Each vertex is in at most $\alpha$ paths $P_{i,j}$ and in at most one path $P_i$. Thus each vertex appears in at most $\alpha+1$ walks $W_e$.
Each walk $W_e$ might appear in at most two sets of the bramble, so the overall congestion is at most $2\alpha+2$.
\end{proof}

\begin{figure}
\includegraphics[scale=1,page=1]{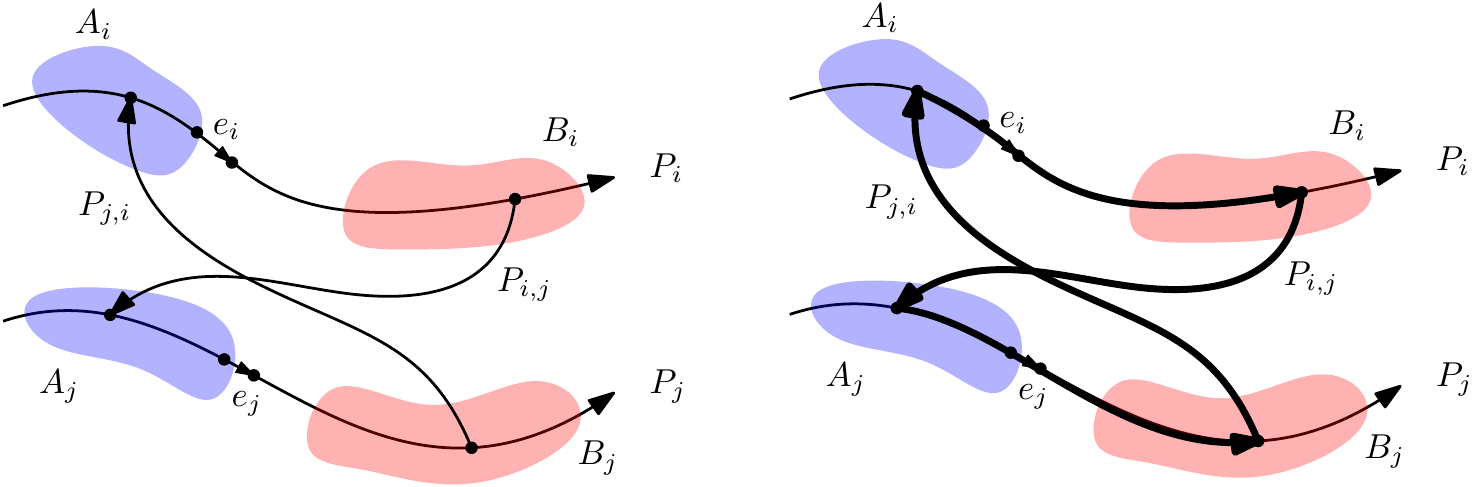}
\caption{The walk $W_e$ constructed in the proof of~\cref{lem:sparse1}.}\label{fig:walkWe}
\end{figure}

\begin{lemma}[Sparse winning scenario, wrapped]\label{lem:sparse2}
Let $c$ be the constant from~\cref{lem:sparse1}.
Let $a > 1$, $b \geq 1$, and let $(P_i,A_i,B_i)_{i=1}^a$ be an $(a,b)$-path system in $G$.
Let $\mathcal{I}$ be a subset of $[a] \times [a] \setminus \{(i,i) \mid i \in [a]\}$, such that $|\mathcal{I}| \geq 0.6 \cdot a(a-1)$.
For $(i,j) \in \mathcal{I}$, let $\LL_{i,j}$ be a linkage of size $b$ from $B_i$ to $A_j$.
Assume that for some integer $d$, the intersection graph of $\LL_{i,j}$ and $\LL_{i',j'}$ for every distinct $(i,j), (i',j') \in \mathcal{I}$
is $d$-degenerate. 
If $b > 4\cdot e \cdot a^2 \cdot d$, then $G$ contains a bramble of congestion at most $4$ and size at least $c \cdot  \left(\frac{a^{1/2}}{\log^{1/4} a}\right)$.
\end{lemma}
\begin{proof}
Construct an auxiliary graph $H$, whose vertices are paths in $\bigcup_{(i,j) \in \mathcal{I}} \LL_{i,j}$.
Two paths are adjacent in $H$ if they contain a common vertex.
Note that each set $\LL_{i,j}$ is independent in $H$, so $H$ is $|\mathcal{I}|$-partite.
Furthermore, as for each  $(i,j), (i',j') \in \mathcal{I}$, the graph $H[ \LL_{i,j} \cup  \LL_{i',j'}]$ is precisely the intersection graph of $\LL_{i,j}$ and $\LL_{i',j'}$, it is $d$-degenerate.
Finally, since $\mathcal{I} \subseteq [a] \times [a]$, we have
\begin{align*}
b & \geq 4 \cdot e \cdot a^2 \cdot d \geq 4 \cdot e \cdot (|\mathcal{I}|-1) \cdot d.
\end{align*}
Thus, applying~\cref{lem:LLL} to $H$ yields a single path $P_{i,j} \in \LL_{i,j}$ for each $(i,j) \in \mathcal{I}$,
such that the paths in $\{ P_{i,j} \}_{(i,j) \in \mathcal{I}}$ are pairwise disjoint.

Now we observe that the set $\mathcal{I}$ and the family $\{ P_{i,j} \}_{(i,j) \in \mathcal{I}}$ of paths satisfy the assumptions of~\cref{lem:sparse1} with $\alpha=1$. Thus the application of~\cref{lem:sparse1} yields a desired bramble.
\end{proof}

\subsection{Closed Walks and Threaded Linkages}\label{sec:threaded}
In this section, we introduce a key object; the threaded linkage that we will use as a main building block in our main proof.
Informally, we want to order and connect paths within the linkage.
In order to achieve it, we construct one long walk, which contains all the paths from the linkage interconnected by walks denoted as threads.
Our ultimate goal is to find a collection of closed walks, each containing a path from a linkage or, in case of two linkages, a collection of closed walks, each containing a path from both linkages.
The latter outcome is provided by another basic tool: Bowtie lemma, which might be useful on its own.
This concept was essentially proved and used in~\cite{DBLP:conf/esa/MasarikMPRS19,DBLP:journals/corr/abs-1907-02494} in a slightly different setting.
We will describe the differences later.

\begin{definition}[Threaded linkage]\label{def:threaded}
A \emph{threaded linkage} is a pair $(W,\LL)$ where $\LL = \{L_1,L_2,\ldots,L_\ell\}$ is a linkage
and $W$ is a walk such that there exist $\ell - 1$ paths $Q_1, Q_2, \ldots, Q_{\ell - 1}$ such that $W$ is the concatenation of $L_1,Q_1, L_2, Q_2, \ldots,Q_{\ell-1},L_\ell$ in that order.
The paths $Q_i$ are called \emph{threads}.
A threaded linkage $(W,\LL)$ for $W = (L_1,Q_1,\ldots,Q_{\ell-1},L_\ell)$ is \emph{untangled} if for every $i$, the thread $Q_i$ may only intersect the rest of $W$ in $L_i$ or $L_{i+1}$.
\end{definition}

The \emph{size} of an (untangled) threaded linkage $(W,\LL)$ is the size of linkage $\LL$ and its \emph{overlap} is the overlap of the walk $W$.

\begin{lemma}[Construction of threaded linkages]\label{lem:Wij}
Let $(P_i,A_i,B_i)_{i=1}^a$ be $(a,b)$-path system for $a,b\in \mathbb{N}$.
Then, for all $i,j \in [a]$, there exists a linkage $\LL_{i,j}$ from $B_i$ to $A_j$ and threaded linkage $(W_{i,j},\LL_{i,j})$ of size $b$ and overlap at most 3.
\end{lemma}
\begin{proof}
We construct a threaded linkage for each $i,j\in[a]$ separately.
For every $i, j \in [a]$, we fix a linkage $\LL_{i,j}$ from $B_i$ to $A_j$ and a linkage $\RLL_{i,j}$ from $A_j$ to~$B_i$; these linkages exist by well-linkedness of $\bigcup_{i = 1}^{a}A_i \cup B_i$.

For every $P \in \LL_{i,j}$ let $\rho_{i,j}(P)$ be the path of $\RLL_{i,j}$ that starts at the ending point of $P$ and let $\pi_{i,j}(P)$ be the path of $\LL_{i,j}$ that starts at the ending point of $\rho_{i,j}(P)$. Note that $\pi_{i,j}$ is a permutation of $\LL_{i,j}$.
Let $\cycles_{i,j}$ be the family of cycles of the permutation $\pi_{i,j}$, observe that every such a cycle corresponds to a closed walk composed of the paths in $\LL_{i,j}$ and $\RLL_{i,j}$.

From every cycle $C \in \cycles_{i,j}$ we arbitrarily select one path; we call it the \emph{representative} of $C$.
Let $C_1,C_2,\ldots,C_r$ be the elements of $\cycles_{i,j}$ in the order of the appearance of the starting points of their representatives along $P_i$. 
Define the walk $W_{i,j}$ as follows: follow $P_i$ and for every $\ell \in [r]$, when we encounter the starting point of the representative of $C_\ell$, follow the respective closed walk corresponding to $C_\ell$,
returning back to the starting point of the representative of $C_\ell$, and then continue going along $P_i$.
Finally, trim $W_{i,j}$ so that it starts and ends with a path of $\LL_{i,j}$, as required by the definition of a threaded linkage.
  
Recall that the size of $(W_{i,j},\LL_{i,j})$ is the size of the linkage $\LL_{i,j}$, i.e., $b$.
Now let us argue about the overlap. The walk $W_{i,j}$ consists of the following subwalks: (1) paths of $\LL_{i,j}$ (each path is used exactly once), (2) paths of $\RLL_{i,j}$ (each path is used at most once), and (3) some pairwise vertex-disjoint subpaths of~$P_i$.
Note that the subwalks within each of these three groups are vertex-disjoint. Thus the overlap of $(W_{i,j},\LL_{i,j})$ is at most 3.
\end{proof}

Now, we refine the threaded linkage to get at least one good outcome: either a collection of closed walks, each containing a path from the linkage or an untangled threaded linkage.

\begin{lemma}[Construction of closed walks or untangled threaded linkages]\label{lem:Wij2}
 Let $(W,\LL)$ be a threaded linkage of size $b$ and of overlap $\alpha$.
 Let $x, d \in \mathbb{N}$ such that $b\ge xd+(d-1)$.
 Then one of the following exists:
\begin{enumerate}
\item A family $\mathcal{Z}$ of $d$ closed walks, such that for every walk $W \in \mathcal{Z}$
  there exists a distinct path $P(W) \in \LL$ that is a subwalk of $W$, and $\mathcal{Z}$ has overlap $\alpha$; or
\item an untangled threaded linkage $(W', \LL')$ where $W'$ is a subwalk $W$ and $\LL' \subseteq \LL$ is of size at least $x$. 
In particular, $(W', \LL')$ is of overlap $\alpha$. 
\end{enumerate}
\end{lemma}
\begin{proof}
Let  $z$ be the length of $W$ (i.e., the number of occurrences of vertices).
For $1 \leq p \leq q \leq z$, by $W[p]$ we will denote the $p$-th vertex of $W$ and by $W[p,q]$ we denote the subwalk $W[p],W[p+1],\ldots,W[q]$.

A \emph{useful walk of $W$} is a subwalk $W[p,q]$ of $W$,
such that $W[p] = W[q]$ and $W[p,q]$ contains at least one path of $\LL$ as a subwalk.
The pair of indices $(p,q)$ is called a \emph{useful intersection}.

We greedily construct a sequence $I_1,I_2,\ldots,I_\ell$ of useful walks as follows:
$I_1 = W[p_1,q_1]$ is a useful walk of $W$ such that $q_1$ is the smallest possible,
and subsequently $I_{\xi+1} = W[p_{\xi+1},q_{\xi+1}]$ is a useful walk of $W$ such that $p_{\xi+1} > q_\xi$ and $q_{\xi+1}$ is the smallest possible.
The greedy construction stops when there are no useful walks starting after $q_\ell$.

First, consider the case that $\ell \geq d$.
Then every useful walk $I_\xi$ is a closed walk and, as $W$ is of overlap $\alpha$, the family
$\{I_\xi \mid \xi \in [\ell]\}$ is of overlap $\alpha$. 
Furthermore, by the definition of a useful walk, every $I_\xi$ contains a distinct path from $\LL$, so we obtain the first desired outcome.

So now consider the case that $\ell < d$.
We select $\ell+1$ subwalks $I'_1,I'_2,\ldots,I'_{\ell+1}$ in $W$ as follows.
The subwalk $I_1'$ is defined as $W[1,q_1-1]$.
Then, for $2 \leq \xi \leq  \ell$, we define $I_\xi'$ as $W[q_{\xi-1}+1, q_{\xi}-1]$.
Finally, we define $I'_{\ell+1} := W[q_{\ell}+1,z]$.

By the construction of the walks $I_\xi$, no $I_\xi'$ contains a useful walk.
Furthermore, the union of all walks $I_\xi'$ covers $W_{i,j}$, except for $q_1,q_2,\ldots,q_\ell$.
Hence, for at least $|\LL| - \ell$ paths $P \in \LL$ it holds there is $\xi \in [\ell-1]$ such that $P$ is fully contained in $I_\xi'$.
So there is some $\xi \in [\ell+1]$, such that $I_\xi'$ contains at least
\[
  \frac{|\LL| - \ell}{\ell+1} ~\geq~  \frac{b - (d-1)}{d} ~\geq~ x
\]
paths of $\LL$, where the last inequality holds because $b \geq xd + (d - 1)$ by precondition.
Let $\LL' \subseteq \LL$ be the set of paths contained in $I'_\xi$ and let $W'$ be the walk $I_\xi'$, trimmed so that it starts and ends with a path of $\LL'$. We note that $(W',\LL')$ is an untangled threaded linkage as $I_\xi'$ contains no useful intersection.
Thus, in this case, we obtain the second desired outcome.
\end{proof}

\begin{lemma}[Bowtie lemma]\label{lem:bowtie}
For $d \geq 1$, let $(W_1,\LL_1)$ and $(W_2,\LL_2)$ be two threaded linkages of overlap $\alpha$ and $\beta$, respectively,
 such that the intersection graph $I(\LL_1,\LL_2)$   of $\LL_1$ and $\LL_2$ is not $(2^9\cdot 5\cdot d)$-degenerate.
Then there is a family $\mathcal{Z}$ of $d$
closed walks such that every walk in $\mathcal{Z}$ contains at least 
one path of $\LL_1$ and one path of $\LL_2$ as a subwalk,
and the congestion of $\mathcal{Z}$ is at most $\alpha+\beta$.

\noindent Furthermore,
if $(W_1,\LL_1)$ ($(W_2,\LL_2)$, respectively) is untangled, then $\Zz$ is of congestion at most $\beta+1$ ($\alpha+1$, respectively),
and if both $(W_1,\LL_1)$, $(W_2,\LL_2)$ are untangled, then $\Zz$ is of congestion at most $2$.
\end{lemma}

The main difference with version proven implicitly in~\cite{DBLP:conf/esa/MasarikMPRS19,DBLP:journals/corr/abs-1907-02494} is that there only the containment of a subpath of a path in $\LL_1$ and of a subpath of a path in $\LL_2$ was guaranteed as opposed to the whole path we provide in the statement.
For the proof, we make use of the following Partitioning Lemma.

\begin{lemma}[Partitioning Lemma {\cite[Lemma 11]{DBLP:journals/corr/abs-1907-02494}}] \label{lem:partition}
Let $k, r \geq 1$ be two integers and let $G$ be a bipartite graph with bipartition classes $X = \{x_1,x_2,\ldots,x_a\}$ and $Y=\{y_1,y_2,\ldots,y_b\}$ and minimum degree at least $2^9 \cdot r \cdot k$.
Then there are $k$ sets $U_1,U_2,\ldots,U_k$, and $k$ sets $W_1,W_2,\ldots,W_k$, such that:
\begin{compactenum}
\item for each $i \in [k]$ the set $U_i$ is a segment of $X$ and the set $W_i$ is a segment of $Y$,
\item for each distinct $i,j \in [k]$ we have $U_i \cap U_j = \emptyset$ and $W_i \cap W_j = \emptyset$,
\item for every $i \in [k]$, the average degree of the graph $G[U_i \cup W_i]$ is at least $r$.
\end{compactenum}
\end{lemma}

Now we are ready to prove~\cref{lem:bowtie}.

\begin{proof}[Proof of~\cref{lem:bowtie}]
First, we invoke~\cref{lem:partition} for $k=d$, $r=5$, and the graph $I(\LL_1,\LL_2)$, where the order of paths in each linkage is naturally determined by the ordering of their appearance in the walks $W_1$, $W_2$.
We obtain a partition of $\LL_1^1,\ldots,\LL_1^d$ of $\LL_1$ and a partition $\LL_2^1,\ldots,\LL_2^d$ of $\LL_2$, satisfying the conditions given in the lemma.

Fix some $i \in [d]$.
Recall that $I(\LL^i_1 \cup \LL_2^i)$ is of average degree at least 5, there are $\widetilde{\LL}^i_1 \subseteq \LL^i_1$ and $\widetilde{\LL}^i_2 \subseteq \LL^i_2$, so that the graph $I(\widetilde{\LL}^i_1 \cup \widetilde{\LL}^i_2)$ is of minimum degree at least 3.
Indeed, after a vertex of degree at most 2 is removed, the average degree is still at least 5.

Let $P_1,P_2,\ldots,P_z$ be the paths of $\widetilde{\LL}^i_1$, ordered by their appearance on $W_1$.
Let $R_1,R_2,\ldots,R_{z-1}$ be the walks, such that $R_i$ is the subwalk of $W_1$ between $P_i$ and $P_{i+1}$. Note that $R_i$ might be either a single thread of $(W_1,\LL_1)$, or a subwalk consisting of alternating threads and paths of $\LL^i$ that were not included in $\widetilde{\LL}^i$. 

Similarly we define $P'_1,P'_2,\ldots,P'_{z'}$ as the paths of $\widetilde{\LL}^i_2$ and $R'_1,R'_2,\ldots,R'_{z'-1}$ as the corresponding subwalks of $W_2$.
Finally, let us denote the concatenations of subwalks defined above:
\begin{align*}
W_1^i := & P_1,R_1,P_2,R_2,\ldots,R_{z-1},P_z\\
W_2^i := & P'_1,R'_1,P'_2,R'_2,\ldots,R'_{z'-1},P'_{z'}.
\end{align*}
Since $P_z$ is of degree at least 3 in $I(\widetilde{\LL}_1^i,\widetilde{\LL}_2^i)$, there exists a common vertex $v$ of $P_z$ and some $P'_{q'}$ for $q'<z'-1$.
Symmetrically, there is a common vertex $w$ of $P'_{z'}$ and $P_q$ for some $q<z-1$.
Consequently, we construct a closed walk, see also~\cref{fig:walkZi}: 
\[Z_i := v,\text{the rest of } P'_{q'},R'_{q'},P'_{q'+1},\ldots,w,\text{the rest of }P_{q},R_{q},P_{q+1},\ldots,v.\]
Note that $Z_i$ fully contains at least one path from each linkage, i.e., $P_{q+1} \in \LL_1$ and $P'_{q'+1} \in \LL_2$.

\begin{figure}
\includegraphics[scale=1,page=2]{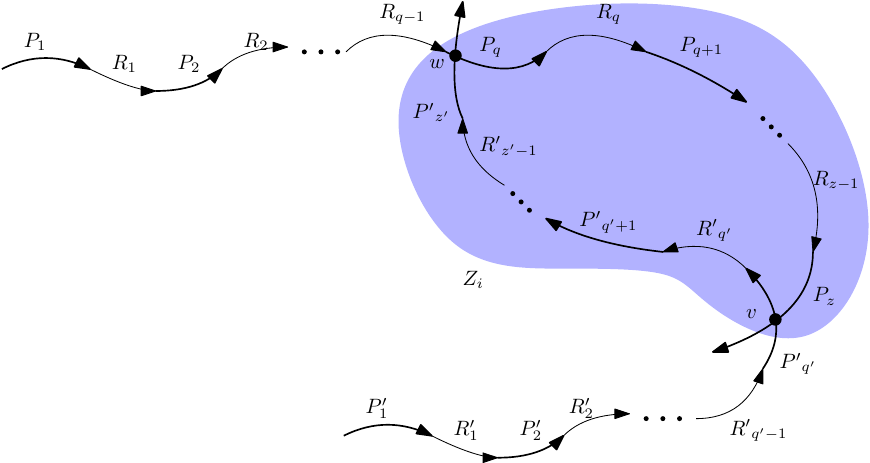}
\caption{The walk $Z_i$ constructed in the proof of~\cref{lem:bowtie}.}\label{fig:walkZi}
\end{figure}

To see that the family $\mathcal{Z} = \{ Z_i \mid i \in [d]\}$ satisfies the statement of the lemma, it remains to discuss the congestion.
By condition 2 of~\cref{lem:partition} each vertex can appear at most $\alpha$ times on $W_1$, thus is contributes as a part of at most $\alpha$ walks $W^i_1$.
Similarly, each vertex contributes as a part of at most $\beta$ walks $W^i_2$. Summing up, each vertex may appear in at most $\alpha + \beta$ elements of $\mathcal{Z}$.

Now, suppose that one of the input threaded linkages, say, $(W_1,\LL_1)$, is untangled.
Let us enumerate $\LL_1 = \{L_1,L_2,\ldots,L_\ell\}$ and let $Q_1,Q_2,\ldots,Q_{\ell-1}$ be the threads of $(W_1,\LL_1)$ as in~\cref{def:threaded}.
Recall that the paths from $\LL_1$ are vertex-disjoint. Now consider a thread $Q_j$.
Note at most one walks $W_1^i$ might contain $Q_j$ as a subwalk.
Furthermore, if $Q_j$ is a subwalk of $W_1^i$, then so are $L_j$ and $L_{j+1}$.
Thus, since $(W_1,\LL_1$) is untangled, each vertex from $W_1$ might contribute to the congestion of at most one walk in $Z \in \Zz$.
\end{proof}

\section{Main proof}\label{sec:main}
Using the tools from \cref{sec:tools} we now prove \cref{thm:main}.
As mentioned before, the starting point is a path system (see \cref{def:path-system}) containing a large number of sets $A_i, B_i$, $i \in [a]$, and large linkages $\LL_{i, j}$ between pairs of such sets.
The basic goal is to exploit the interplay between the dense and sparse winning scenarios from \cref{sec:brambles}.
The dense winning scenario (\cref{lem:dense}) is applicable if there is a pair of linkages whose intersection graph has large degeneracy.
The crux is how to apply the sparse winning scenarios (\cref{lem:sparse1,lem:sparse2}) because they need a fraction of pairs of linkages that is slightly larger than half of the available pairs of linkages.
Our goal is thus to distinguish three subsets of the set of pairs of linkages such that (i) one of the subsets will be larger than half the available pairs of linkages and (ii) each of the subsets can be used to define a bramble of large size.
We obtain these three distinguished subsets by using two partially overlaying matchings in auxiliary graphs whose vertex sets are the available pairs of linkages and whose edges represent the degeneracy of the corresponding intersection graphs.

Some complications arise from the aim of keeping the congestion low.
To achieve this, instead of applying the sparse winning scenarios directly to the linkages from the path set system, we use first the tools from \cref{sec:threaded} to obtain subsets of pairs of linkages of lower congestion.

\paragraph{Setup.}
Let \(k \in \mathbb{N}\) with \(k > 1\) and let $G$ be a graph of directed treewidth $t$.
We show that if $t \geq c_t \cdot k^{48} \log^{13} k$, then $G$ contains a bramble of congestion at most~8 and size~$k$.
Herein, $c_t$ is a constant that we specify below.

We start by fixing the following parameters. Let $c$ be the constant in \cref{lem:sparse2}. Without loss of generality we can assume that $c \leq \frac{1}{3^{1/4}}$ and thus $c^{-4} \geq 3$; this will be used in the case analysis later in the proof.
We put $c_a = c^{-4} \geq 3$.
Let $c_{d_3} := c_{\textsf{KT}}$ be the constant from \cref{thm:kostochka}. We define:
\begin{align*}
  a &= \left\lceil c_a \cdot k^2 (1 + \log k)^{1/2} \right\rceil, \\
  d_3 &= \left\lceil c_{d_3} \cdot k \sqrt{\log k} \right\rceil, \\
  d_2 &=  \left\lceil 2^{11}5e \cdot a^2 d_3 \right\rceil = \Oh(k^5 \log^{3/2} k), \\
  d_1 &=  \left\lceil 2^{11}5e \cdot a^2 d_2 \right\rceil = \Oh(k^9 \log^{5/2} k)\text{, and} \\
  b & = \left\lceil 4e \cdot a^2 d_1^2 \right\rceil = \Oh(k^{22} \log^{6} k). 
\end{align*}
Let $c_{\textsf{KK}}$ be the constant in \cref{lem:KK}.
Choose the constant~$c_t$ such that $t \geq c_{\textsf{KK}} a^2 b^2$.
To see that this is possible, observe that
\begin{align*}
  c_{\textsf{KK}} a^2 b^2 \leq & \;  2^5 e^2 \cdot c_{\textsf{KK}} \cdot a^6 d_1^4 \\
  \leq & \; 2^{50} 5^4 e^6 \cdot c_{\textsf{KK}} \cdot a^{14} d_2^4\\
  \leq & \; 2^{95} 5^{8} e^{10} \cdot c_{\textsf{KK}} \cdot a^{22}d_3^4\\
  \leq & \; 2^{96} 5^{8} e^{10} \cdot c_{\textsf{KK}} \cdot c_{d_3}^4 \cdot a^{22}k^4\log^2 k \\
  \leq & \; 2^{97} 5^{8} e^{10} \cdot c_{\textsf{KK}}  \cdot c_{d_3}^4 \cdot c_a^{22} \cdot k^{48} \log^2 k (1+\log k)^{11}\\
  \leq & \; 2^{97} 5^{8} e^{10} \cdot c_{\textsf{KK}} \cdot c_{d_3}^4 \cdot c_a^{22} \cdot k^{48} \log^2 k (2\log k)^{11}\\
  \leq & \; 2^{108} 5^{8} e^{10} \cdot c_{\textsf{KK}} \cdot c_{d_3}^4 \cdot c_a^{22} \cdot k^{48} \log^{13} k.
\end{align*}
Thus, putting $c_t = 2^{108} 5^{8} e^{10} \cdot c_{\textsf{KK}} \cdot c_{d_3}^4c_a^{22}$, we have $t \geq c_{\textsf{KK}} a^2b^2$.
Hence, by \cref{lem:KK}, there is an $(a,b)$-path system $(P_i,A_i,B_i)_{i=1}^a$ in~$G$.

\paragraph{Large subsets of pairs of linkages and sets of closed walks of low congestion.}
Let $V = \{(i,j) \mid i, j \in [a] \wedge i \neq j\}$.
Our aim is now to find the three subsets of $V$ such that one of them will be larger than $|V|/2$ mentioned above and to subset some linkages in order to achieve low congestion.
We achieve the second aim by using sets of closed walks derived from some of the linkages.
See \cref{tb:families} for the various linkages and sets of walks that we define below and their properties.

\begin{figure}[tb]
\begin{center}
\begin{tabular}[t]{|l|l|c|c|l|}
\hline \textbf{Index} & \textbf{Family} & \textbf{o/c} & \textbf{Size} & \textbf{Comments} \\\hline
$(i,j) \in V$ & $(W_{i,j}, \LL_{i,j})$ & $\mathrm{o} \leq 3$ & $b$ & threaded linkage \\\hline
$(i,j) \in Z$ & $\mathcal{Z}_{i,j}$ & $\mathrm{o} \leq 3$ & $d_1/(2^9 \cdot 5)$ & \makecell[l]{closed walks \\ each $W \in \mathcal{Z}_{i,j}$ contains $P_{i,j}(W) \in \LL_{i,j}$} \\\hline
$(i,j) \in V \setminus Z$ & $(W_{i,j}', \LL_{i,j}')$ & $\mathrm{o} \leq 3$ & $4ea^2 d_1 + 1$ & untangled threaded linkage \\\hline
$e \in M_1$ & $\mathcal{Z}_e$ & $\mathrm{c} \leq 2$ & $d_1/(2^9 \cdot 5)$ & \makecell[l]{closed walks, for $(i,j) \in e$ \\ each $W \in \mathcal{Z}_e$ contains $P_{i,j}(W) \in \LL_{i,j}$}
  \\\hline
$e \in M_2$ & $\mathcal{Z}_e$ & $\mathrm{c} \leq 4$ & $d_2/(2^9 \cdot 5)$ & \makecell[l]{closed walks, for $(i,j) \in e$ \\ each $W \in \mathcal{Z}_e$ contains $P_{i,j}(W) \in \LL_{i,j}$} \\\hline
$\begin{array}{l} F \subseteq Z \cup M_1 \cup M_2 \\ 1 \leq |F| \leq 2 \end{array}$ & $\mathcal{Z}_F$ & $\mathrm{c} \leq 8$ & $d_2/(2^9 \cdot 5)$ & intersection graph $d_3$-degenerate \\\hline
$(i,j) \in Z$ & $\LL_{i,j}^\mathcal{Z}$ & $\mathrm{c} \leq 1$ & $d_1/(2^9 \cdot 5)$ & \makecell[l]{linkage \\ $\LL_{i,j}^\mathcal{Z} = \{ P_{i,j}(W)~|~W \in \mathcal{Z}_{i,j}\}$}\\\hline
$(i,j) \in V(M_1)$ & $\LL_{i,j}^\mathcal{Z}$ & $\mathrm{c} \leq 1$ & $d_1/(2^9 \cdot 5)$ & \makecell[l]{linkage, for $e \in M_1$ with $(i,j) \in e$ \\ $\LL_{i,j}^\mathcal{Z} = \{ P_{i,j}(W)~|~W \in \mathcal{Z}_e\}$} \\\hline
$\begin{array}{l}
  (i,j) \in V(M_2) \\
  \quad\setminus (Z \cup V(M_1)) \end{array} $
& $\LL_{i,j}^\mathcal{Z}$ & $\mathrm{c} \leq 1$ & $d_2/(2^9 \cdot 5)$ & \makecell[l]{linkage, for $e \in M_2$ with $(i,j) \in e$ \\ $\LL_{i,j}^\mathcal{Z} = \{ P_{i,j}(W)~|~W \in \mathcal{Z}_e\}$} \\\hline
\end{tabular}
\caption{Important linkages and families of closed walks. ``o'' stands for overlap and ``c'' stands for congestion.}\label{tb:families}
\end{center}
\end{figure}

We apply \cref{lem:Wij} to $(P_i,A_i,B_i)_{i=1}^a$, obtaining for every $(i,j) \in V$ a threaded linkage~$(W_{i, j}, \LL_{i, j})$ of size $b$ and overlap at most 3.
Then, we apply~\cref{lem:Wij2} to $(W_{i, j}, \LL_{i, j})$ with $x = 4ea^2d_1 + 1$ and $d = d_1/(2^9 \cdot 5)$.
Note that $xd  + (d - 1) 
\leq 4e \cdot a^2d_1^2/(2^9 \cdot 5) + d_1/(2^8 \cdot 5) \leq 4e \cdot a^2d_1^2 \leq b$.
Hence, the preconditions of \cref{lem:Wij2} are satisfied.
Let $Z \subseteq V$ be the set of pairs $(i,j)$ for which the application of~\cref{lem:Wij2} results in the first outcome.

For each $(i, j) \in V \setminus Z$: Let $(W_{i,j}',\LL_{i,j}')$ be the \emph{untangled} threaded linkage resulting from the application of~\cref{lem:Wij2}.
Note that $|\LL'_{i, j}| \geq x = 4e \cdot a^2d_1 + 1$.

For each $(i, j) \in Z$, that is, for each $(i, j)$ where applying~\cref{lem:Wij2} results in the first outcome: Let $\mathcal{Z}_{i,j}$ be the family of closed walks resulting from the application of~\cref{lem:Wij2}.
Observe that $\mathcal{Z}_{i,j}$ is of size at least $d_1/(2^9 \cdot 5)$ and of overlap at most~$3$.
By the definition of $\mathcal{Z}_{i,j}$, for each walk $W \in \mathcal{Z}_{i,j}$ there is a distinct path $P(W) \in  \LL_{i, j}$ such that $P(W)$ is a subwalk of~$W$.
Define the linkage $\LL_{i,j}' := \{P(W) \mid W \in \mathcal{Z}_{i,j}\}$.
For convenience, we denote also $W_{i,j}' := W_{i,j}$.
Observe that $(W_{i,j}', \LL_{i,j}')$ is a threaded linkage (but not necessarily untangled).

For both $\ell = 1,2$, let $E_\ell \subseteq \binom{V}{2}$ be the set of those pairs $\{(i,j), (i',j')\} \in \binom{V}{2}$, for which the intersection graph of $\LL_{i,j}'$ and $\LL_{i',j'}'$ is not $d_\ell$-degenerate.
Define an undirected graph $H_\ell = (V,E_\ell)$.
Since $d_1 \geq d_2$, we have $E_1 \subseteq E_2$, and thus $H_1$ is a subgraph of $H_2$.

Let $M_1$ be a maximum matching in $H_1 - Z$.
Let $M_2$ be a maximum matching in the graph $(V, E(H_2) \setminus \binom{V(M_1) \cup Z}{2})$, that is,
in the graph that results from $H_2$ by removing all edges with both endpoints in $V(M_1) \cup Z$, see~\cref{fig:zm1m2}.

\begin{figure}[t]
  \centering
  \includegraphics{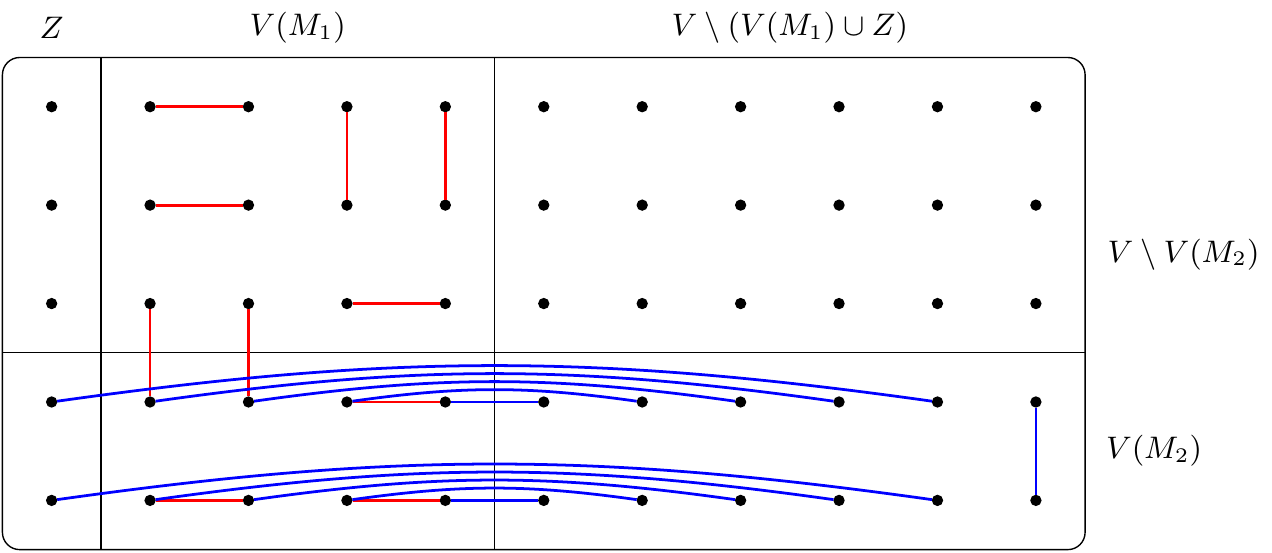}
  \caption{The relation of $V$, $Z$, and the matchings $M_1$ (red) and $M_2$ (blue).}\label{fig:zm1m2}
\end{figure}
\begin{figure}[t]
  \centering
  \includegraphics{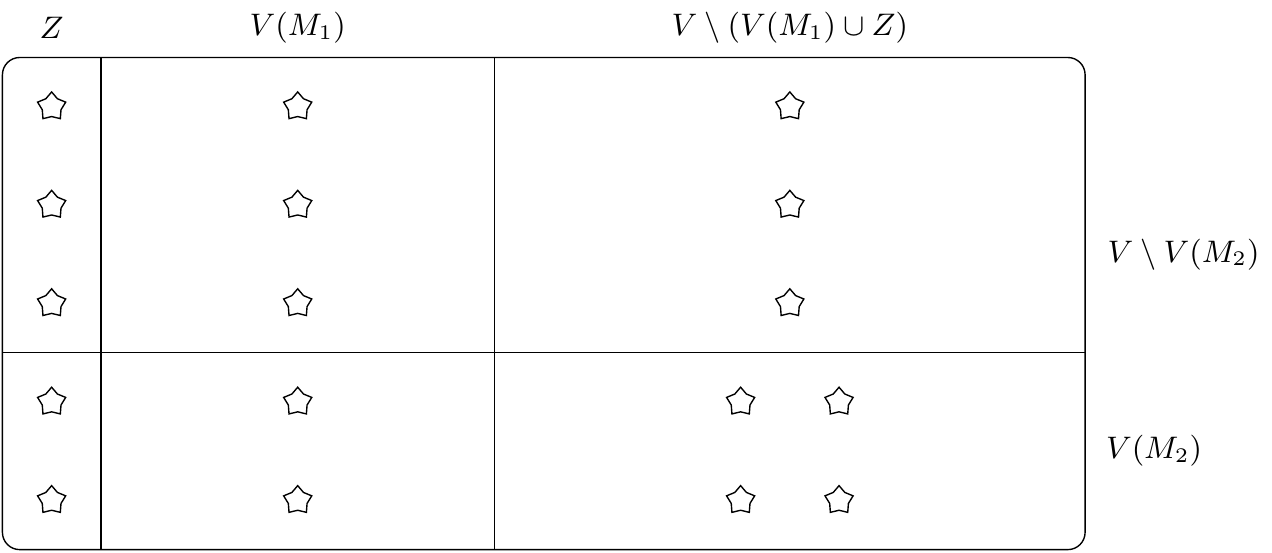}
  \caption{Proof of~\cref{lem:cases}: the number of stars in each region corresponds to the number of times the set is counted in the proof of Claim~\ref{lem:cases}.}
  \label{fig:cases}
\end{figure}

We are now ready to define the promised three vertex subsets of $V$ such that one of them is guaranteed to be sufficiently large for our purposes.

\begin{claim}\label{lem:cases}
  At least one of the following cases occurs:
  \begin{enumerate}[{Case} 1.]
  \item\label{p:sparse} $|V \setminus (V(M_1) \cup Z)| \geq 0.6 |V|$;
  \item\label{p:dense} $|V(M_1) \cup V(M_2) \cup Z| \geq 0.6 |V|$;
  \item\label{p:mixed} $|V \setminus V(M_2)| \geq 0.6 |V|$.
  \end{enumerate}
\end{claim}
\begin{proof}
  It suffices to show that
  \begin{align}
    2|V \setminus (V(M_1) \cup Z)| + 2|V(M_1) \cup V(M_2) \cup Z| + |V \setminus V(M_2)|  \geq 3|V|.\label{eq:cases1}
  \end{align}
  Consider how often vertices in the following vertex subsets are counted in the left hand side of~\eqref{eq:cases1}; consult also~\cref{fig:cases}:
  \begin{itemize}
  \item Each vertex in $V \setminus (V(M_1) \cup V(M_2) \cup Z)$ is counted thrice.
  \item Each vertex in $V(M_2) \setminus (V(M_1) \cup Z)$ is counted four times.
  \item Each vertex in $V(M_2) \cap (V(M_1) \cup Z)$ is counted twice.
  \item Each vertex in $V(M_1) \setminus V(M_2)$ is counted thrice.
  \item Each vertex in $Z \setminus V(M_2)$ is counted thrice.
  \end{itemize}
  Now note that no vertex of $V$ occurs in two or more of the above sets.
  Hence, the left hand side of~\eqref{eq:cases1} is at least
  \begin{multline*}
    3|V \setminus (V(M_1) \cup V(M_2) \cup Z)| + 4 | V(M_2) \setminus (V(M_1) \cup Z)| \\
    + 2|V(M_2) \cap (V(M_1) \cup Z)| + 3|V(M_1) \setminus V(M_2)| + 3|Z \setminus V(M_2)|.
  \end{multline*}
  Observe that $|(V(M_1) \cup Z) \cap V(M_2)| \leq |V(M_2) \setminus (V(M_1) \cup Z)|$, because every edge of $M_2$ has at most one endpoint in $V(M_1) \cup Z$.
  Thus, the left hand side of~\eqref{eq:cases1} is at least 
  \begin{multline*}
    3|V \setminus (V(M_1) \cup V(M_2) \cup Z)| + 3 | V(M_2) \setminus (V(M_1) \cup Z)| \\
    + 3 |V(M_2) \cap (V(M_1) \cup Z)| + 3|V(M_1) \setminus V(M_2)| + 3|Z \setminus V(M_2)|,
  \end{multline*}
  which is equal to $3|V|$ (recall here that $V(M_1) \cap Z = \emptyset$), as claimed. 
\end{proof}

We now continue with definitions of sets of closed walks and sublinkages that we need to derive a bramble of low congestion for that subset above that is large.
Refer again to Figure~\ref{tb:families} for a summary of all important linkages and families of closed walks defined here.

For both $\ell = 1,2$ and for each $e = \{(i,j), (i',j')\} \in M_\ell$, apply~\cref{lem:bowtie} to $(W_{i,j}',\LL_{i,j}')$ and $(W_{i',j'}',\LL_{i',j'}')$ to find a family $\mathcal{Z}_e$ of $d_\ell/(2^9 \cdot 5)$ closed walks, such that every walk in $\mathcal{Z}_e$ contains at least one path from $\LL'_{i,j}$ and at least one path from $\LL'_{i',j'}$.
Since the matchings $M_1$ and $M_2$ are disjoint, we obtain $|M_1| + |M_2|$ families of closed walks $\mathcal{Z}_e$ for $e \in M_1 \cup M_2$.

Let us now analyze the congestion of the families $\mathcal{Z}_e$.
Recall that for all $(i,j) \notin Z$, the threaded linkage $(W_{i,j}',\LL_{i,j}')$ is untangled.
Thus for each $e \in M_1$, both its endpoints correspond to untangled linkages, so by~\cref{lem:bowtie},
the congestion of $\mathcal{Z}_e$ is at most two.
Now consider the family~$Z_e$ for an edge $e \in M_2$.
At least one endpoint~$(i, j)$ of $e$ is in $V \setminus Z$ and hence $(W_{i,j}',\LL_{i,j}')$ is untangled.
Furthermore, for the other endpoint $(i', j')$ of $e$ we have that $\LL_{i',j'}'$ is of congestion at most 3. 
So by~\cref{lem:bowtie}, the congestion of $\mathcal{Z}_e$ is at most $4$.

Now suppose that for some $F \subseteq M_1 \cup M_2 \cup Z$ of size $1$ or $2$, the intersection graph of $\mathcal{Z}_F := \bigcup_{g \in F} \mathcal{Z}_{g}$ is not $d_3$-degenerate.
Recall that for each $(i,j) \in Z$, the congestion of $\mathcal{Z}_{i,j}$ is at most~3 (since it resulted from the first outcome of \cref{lem:Wij2}), and for each $e \in M_1 \cup M_2$ the congestion of $\mathcal{Z}_{e}$ is at most~$4$.
Thus the congestion of $\mathcal{Z}_F$ is at most~$8$. 
Applying~\cref{lem:dense} to $\mathcal{Z}_F$ thus yields bramble of size~$k$ and congestion at most~8, finishing the proof in this case.
Thus, henceforth the following claim holds.
\begin{claim}\label{cl:zf-degen}
  For each $F \subseteq M_1 \cup M_2 \cup Z$ of size $1$ or $2$ the intersection graph of $\mathcal{Z}_F$ is $d_3$\nobreakdash-degenerate.
\end{claim}

For every $e \in M_1$ and endpoint $(i,j) \in e$ proceed as follows.
Recall that every $W \in \mathcal{Z}_e$ contains a path in $\LL_{i,j}'$.
For every $W \in \mathcal{Z}_e$ let $P_{i,j}(W) \in \LL_{i,j}'$ be an arbitrary such path.
Let $\LL_{i,j}^\mathcal{Z} = \{P_{i,j}(W) \mid W \in \mathcal{Z}_e\}$.
Note that $|\LL_{i,j}^\mathcal{Z}| \geq |\mathcal{Z}_e| \geq d_1/(2^9 \cdot 5)$.

Similarly, for every $e \in M_2$ and endpoint $(i,j) \in e \setminus (V(M_1) \cup Z)$, proceed as follows.
For every $W \in \mathcal{Z}_e$, pick a path $P_{i,j}(W) \in \LL_{i,j}'$ on $W$.
Let $\LL_{i,j}^\mathcal{Z} = \{P_{i,j}(W) \mid W \in \mathcal{Z}_e\}$.
Note that $|\LL_{i,j}^\mathcal{Z}| \geq |\mathcal{Z}_e| \geq d_2/(2^9 \cdot 5)$.

Furthermore, for every $(i,j) \in Z$ set $\LL_{i,j}^\mathcal{Z} = \LL_{i,j}'$.
Note that $|\LL_{i,j}^\mathcal{Z}| \geq |\mathcal{Z}_{i, j}| \geq d_1/(2^9 \cdot 5)$.
Recall that the application of~\cref{lem:Wij2} for $(W_{i, j}, \LL_{i, j})$ in the beginning resulted in the first outcome and thus for every $W \in \mathcal{Z}_{i,j}$ there is a distinct path $P(W) \in \LL_{i, j} \subseteq \LL_{i, j}' = \LL_{i,j}^\mathcal{Z}$ such that $P(W)$ is a subwalk of~$W$.
For every $W \in \mathcal{Z}_{i,j}$ denote $P_{i,j}(W) = P(W) \in \LL_{i,j}'$.

\paragraph{Case distinction.}\label{sec:cases}
We are now ready to deal with the possible outcomes of \cref{lem:cases} one-by-one.
Refer to Figure~\ref{tb:families} to recall the properties of the linkages and families of closed walks used below.

\paragraph{Case~\ref{p:sparse}: \boldmath $|V \setminus (V(M_1) \cup Z)| \geq 0.6 |V|$  (large independent set in $H_1$).}
We would like to apply \cref{lem:sparse2} to the path system $(P_i, A_i, B_i)_{i = 1}^a$ with $\mathcal{I} := V \setminus (V(M_1) \cup Z)$.
To apply \cref{lem:sparse2} we check that (i)~$|\mathcal{I}| \geq 0.6 \cdot a \cdot (a - 1)$, which is true since $|\mathcal{I}| \geq 0.6 |V| = 0.6 \cdot a \cdot (a - 1)$, that (ii) for every $(i, j) \in \mathcal{I}$ there is a linkage of size at least $4e \cdot a^2d_1 + 1$ between points in $A_i$ and $B_j$, that (iii) for every two $(i, j), (i', j') \in \mathcal{I}$ the intersection graph of the two linkages is $d_1$-degenerate and that (iv) the size of the linkages is strictly larger than $4 e \cdot a^2 d_1$, which clearly holds by definition.

For the linkages in point (ii) we choose the linkages $\LL'_{i, j}$.
Observe that, since each $(i, j) \in \mathcal{I}$ is not in $Z$, we have $|\LL'_{i, j}| > 4e \cdot a^2d_1$, as required by point~(ii).
Since $M_1$ is a maximum matching in $H_1-Z$ it follows that $\mathcal{I}$ is an independent set in~$H_1$.
By the definition of~$H_1$ and $M_1$, for every two distinct pairs $(i,j), (i',j') \in \mathcal{I}$, the intersection graph of $\LL_{i,j}'$ and $\LL_{i',j'}'$ is thus $d_1$-degenerate.
Thus, point~(iii) holds as well, meaning that \cref{lem:sparse2} is applicable.

Recall that $c$ is the constant in \cref{lem:sparse2} and we have $c_a = c^{-4} \geq 3$.
The application of~\cref{lem:sparse2} yields a bramble of congestion at most $4$ and size at least
\begin{align}
  c \cdot \frac{a^{1/2}}{\log^{1/4} a} & \geq c \cdot \frac{(c_a \cdot k^2(1 + \log k)^{1/2})^{1/2}}{\log^{1/4} (c_a k^2(4 + \log k)^{1/2})} 
                                         = k \cdot \frac{cc_a^{1/2} (1 + \log k)^{1/4}}{\log^{1/4} (c_a k^2(4 + \log k)^{1/2})} \nonumber \\
                                       & \geq k \cdot \frac{cc_a^{1/2} (1 + \log k)^{1/4}}{\log^{1/4} (2c_a k^3)}
                                         = k \cdot \frac{cc_a^{1/2} (1 + \log k)^{1/4}}{(\log 2c_a + 3\log k)^{1/4}} \nonumber \\
                                       & \geq k \cdot \frac{c c_a^{1/4} \cdot (c_a + c_a \log k)^{1/4}}{(\log 2c_a + 3\log k)^{1/4}}. \label[equation]{eq:bramble-size-from-lem:sparse2}
\end{align}
As $c_a \geq 3$, we have that $c_a \geq \log 2c_a$ and thus the right-hand side of inequality~\eqref{eq:bramble-size-from-lem:sparse2} is at least~$k$, finishing the proof in this case.

\paragraph{Case~\ref{p:dense}: \boldmath $|V(M_1) \cup V(M_2) \cup Z| \geq 0.6 |V|$ (large matchings).}
We show that we may apply \cref{lem:sparse1} to a set of paths obtained from the families $\mathcal{Z}_g$ for $g \in Z \cup M_1 \cup M_2$.
Let $J$ be the $(|Z|+|M_1|+|M_2|)$-partite graph whose vertex set is the disjoint union of the \emph{color classes} $\{\mathcal{Z}_{g} \mid g \in Z \cup M_1 \cup M_2\}$ such that, for every pair of distinct $g,g' \in Z \cup M_1 \cup M_2$, the graph $J[\mathcal{Z}_g \cup \mathcal{Z}_{g'}]$ is the intersection graph of $\mathcal{Z}_g$ and $\mathcal{Z}_{g'}$ (and $J$ contains no further edges).
Hence we obtain a graph colored with $|Z| + |M_1| + |M_2| \leq |V| = a(a-1)$ colors.
We claim that by \cref{lem:LLL} we may thus select a walk $W_h \in \mathcal{Z}_h$ for every $h \in M_1 \cup M_2 \cup Z$ such that the walks $W_h$ are pairwise vertex-disjoint.
To that end, it suffices to show the required relation on the number of colors of $J$, the size of the color classes, and the degeneracy of subgraphs induced by two color classes.
By \cref{cl:zf-degen} each subgraph of $J$ that is induced by two different colors is $d_3$-degenerate.
Moreover, for every $g \in Z \cup M_1 \cup M_2$ we have $|\mathcal{Z}_g| \geq d_2/(2^9 \cdot 5)$ by definition of $\mathcal{Z}_g$ and since $d_1 > d_2$.
By definition of~$d_2$, we have $d_2/(2^9 \cdot 5) \geq 4e \cdot a^2d_3 \geq 4e(a(a-1) - 1)d_3$, as required by \cref{lem:LLL}.
Thus, indeed, we may choose the walks~$W_h$ as specified.

Let $\mathcal{I}$ be a subset of $Z \cup V(M_1) \cup V(M_2)$ of size exactly $\lceil 0.6 a(a-1) \rceil$.
The set of paths to which we apply \cref{lem:sparse1} is derived from the walks~$W_h$ as follows.
If $h = \{(i,j), (i',j')\} \in M_1 \cup M_2$, then, by definition of $\mathcal{Z}_h$, walk $W_h$ contains a path $P_{i,j}(W_h) \in \LL_{i,j}$ and a path $P_{i',j'}(W_h) \in \LL_{i',j'}$.
If $h = (i,j) \in Z$, then, by definition of $\mathcal{Z}_h$, walk $W_h$ contains a path $P_{i,j}(W_h) \in \LL_{i,j}^\mathcal{Z}$.
Construct a family $\mathcal{Q} = \{P_{i,j} \mid (i,j) \in \mathcal{I}\}$ by, for each $(i, j) \in \mathcal{I}$, choosing an arbitrary walk $W_h$ such that $(i, j) = h$ or $(i, j) \in h$ and putting $P_{i,j} = P_{i, j}(W_h)$.
Note that for each $(i, j) \in \mathcal{I}$ we have $P_{i,j} \in \LL_{i,j}$.
Since two paths $P_{i, j}, P_{i',j'}$ may only share vertices if they stem from the same walk~$W_h$ we have that $\mathcal{Q}$ has congestion at most~$2$.
By applying~\cref{lem:sparse1} to $\mathcal{Q}$, we obtain a bramble of congestion at most $2+2 \cdot 2 = 6$ and of size at least $c \cdot \frac{a^{1/2}}{\log^{1/4} a}$.
By the same calculation as in inequality \eqref{eq:bramble-size-from-lem:sparse2}, this bramble has size at least~$k$, finishing the proof in this~case.

\paragraph{Case~\ref{p:mixed}: \boldmath$|V \setminus V(M_2)| \geq 0.6 |V|$ (large matching anti-adjacent to an independent set).}
We now would like to apply \cref{lem:sparse2} with $\mathcal{I} := V \setminus V(M_2)$ to obtain a bramble of size~$k$ and low congestion.
To apply \cref{lem:sparse2} we check that (i)~$|\mathcal{I}| \geq 0.6 \cdot a \cdot (a - 1)$, which is true since $|\mathcal{I}| = |V \setminus V(M_2)| \geq 0.6 |V| = 0.6 \cdot a \cdot (a - 1)$, that (ii) for every $(i, j) \in \mathcal{I}$ there is a linkage of size $d_1/(2^9 \cdot 5)$ between points in $A_i$ and $B_j$, that (iii) for every two $(i, j), (i', j') \in \mathcal{I}$ the intersection graph of the two linkages is $d_2$-degenerate and that (iv) $d_1/(2^9\cdot 5) > 4 \cdot e \cdot a^2 \cdot d_2$, which clearly holds by definition.

As linkages for point (ii), for each $(i, j) \in \mathcal{I} \cap (Z \cup V(M_1))$ we take the linkage $\LL_{i, j}^\mathcal{Z}$.
Note that $|\LL_{i, j}^\mathcal{Z}| \geq d_1/(2^9 \cdot 5)$.
For each $(i, j) \in \mathcal{I} \setminus (Z \cup V(M_1))$ we take the linkage $\LL_{i, j}'$.
For convenience we denote $\LL_{i, j}^\mathcal{Z} := \LL_{i, j}'$.
Note that $|\LL_{i, j}^\mathcal{Z}| \geq d_1/(2^9 \cdot 5)$ as well.

For point~(iii), observe that, for each $(i, j), (i', j') \in Z \cup V(M_1)$, the intersection graph of the linkages~$\LL_{i, j}^\mathcal{Z}, \LL_{i',j'}^\mathcal{Z}$ is $d_3$-degenerate (and thus $d_2$-degenerate) because the paths in these linkages are contained as subwalks in $\mathcal{Z}_{\{(i, j), (i', j')\}}$ and by \cref{cl:zf-degen}.
For each $(i, j) \in \mathcal{I}$ and $(i', j') \in \mathcal{I} \setminus (Z \cup V(M_1))$ the intersection graph between $\LL_{i, j}^\mathcal{Z}$ and $\LL_{i', j'}^\mathcal{Z}$ is $d_2$-degenerate because otherwise $M_2$ would not be maximum.
Hence also point~(iii) holds, and it follows that \cref{lem:sparse2} is applicable.

From \cref{lem:sparse2} we obtain a bramble of congestion at most~$4$ and size at least $c \cdot \frac{a^{1/2}}{\log^{1/4} a}$, which is at least $k$ by inequality~\eqref{eq:bramble-size-from-lem:sparse2}.

%%% Local Variables:
%%% mode: latex
%%% TeX-master: "main"
%%% End:

\bibliographystyle{plainurl}
\bibliography{../main}
\end{document}